\documentclass[11pt,reqno,a4paper]{amsart}
 
\usepackage{amsmath,amssymb,amsthm,amsaddr}
\usepackage{enumitem}
\usepackage[margin=3cm,top=4cm,bottom=4cm,footskip=1cm,headsep=2cm]{geometry}
\usepackage[utf8]{inputenc}
\usepackage{amsthm}
\usepackage[british]{babel}

\usepackage{accents}

\usepackage{xcolor}

\author{H. Egger$^{1,2}$ \and J. Giesselmann$^{3,\ast}$}
\address{$^1$Institute for Numerical Mathematics, Johannes Kepler University, Linz, Austria\\
$^2$Johann Radon Institute for Computational and Applied Mathematics, Austria\\
$^3$Department of Mathematics, TU Darmstadt, Germany}
\email{herbert.egger@jku.at}
\email{jan.giesselmann@tu-darmstadt.de}

\title[Regularity and long time behavior of a nonlinear parabolic problem]%{Regularity and long time behavior of solutions to a nonlinear parabolic problem and its discretization}
{Regularity and long time behavior of a doubly nonlinear parabolic problem and its discretization}

\newtheorem{theorem}{Theorem}%[section]

\newtheorem{lemma}[theorem]{Lemma}

\newtheorem{problem}[theorem]{Problem}

\theoremstyle{definition}
\newtheorem{remark}[theorem]{Remark}
\newtheorem{assumption}[theorem]{Assumption}

\def\E{\mathcal{E}}
\def\H{\mathcal{H}}
\def\D{\mathcal{D}}

\def\T{\mathcal{T}}

\def\dt{\partial_t}
\def\dx{\partial_x}
\def\ddt{\frac{d}{dt}}
\def\dtau{d_\tau}

\def\la{\langle}
\def\ra{\rangle}

\def\RR{\mathbb{R}}

\begin{document}

\begin{abstract}
   We study a doubly nonlinear parabolic problem arising in the modeling of gas transport in pipelines. Using convexity arguments and relative entropy estimates we show uniform bounds and exponential stability of discrete approximations obtained by a finite element method and implicit time stepping. Due to convergence of the approximations to weak solutions of the problem, our results also imply regularity, uniqueness, and long time stability of weak solutions of the continuous problem.
\end{abstract}

\maketitle

\begin{quote}
{\footnotesize \textbf{Keywords:} 
gas transport, doubly nonlinear parabolic problems, relative entropy estimates, exponential stability, structure preserving discretization

\noindent
\textbf{MSC (2010):} 
%35B35, %   	Stability in context of PDEs [See also 34Dxx, 37B25, 37C20, 37C75, 37F15, 37J25, 37K45, 37L15, 49K40, 58K25, 93Dxx]
35B40, %   	Asymptotic behavior of solutions to PDEs
35K59, %   	Quasilinear parabolic equations
%37L15, %   	Stability problems for infinite-dimensional dissipative dynamical systems
35K67, %   	Singular parabolic equations
65M12 % 	Stability and convergence of numerical methods for initial value and initial-boundary value problems involving PDEs
}
\end{quote}
\bigskip

\section{Introduction}

We consider a doubly nonlinear parabolic equation of the form 
\begin{align}\label{eq:1}
    \dt \beta(u) - \dx ( \mu(\dx u)) &= 0,
\end{align}
with $\beta$ and $\mu$ strictly monotone functions.
Similar to \cite{Bamberger1977,SchoebelKroehn2020}, our motivation for focusing on the one-dimensional setting stems from the modeling of gas transport in pipelines.
In the high friction limit, the balance equations of mass and momentum can be stated as
\begin{align}
    \dt \rho + \dx m &=0\\
    \rho \dx p(\rho) &= - |m| m,
\end{align}
with density $\rho$, mass flux $m$, and pressure law $p(\rho)$;  
see e.g. \cite{BrouwerGasserHerty2011,EggerGiesselmann2020}.
This system can be written in the form \eqref{eq:1}, with $\beta$ denoting the inverse of 
%$\rho \mapsto \rho p'(\rho)$
$\rho \mapsto \int_0^\rho s p'(s) ds$
and $\mu(s)=|s|^{-1/2} s$ and with 
$\rho=\beta(u)$ and $m=\mu(\dx u)$; see \cite{Bamberger1977,DiazThelin1994}.
A fully discrete finite difference approximation  was used in \cite{Raviart1970} to establish existence of weak solutions to \eqref{eq:1} with homogeneous Dirichlet boundary conditions and appropriate initial conditions. 
Uniqueness was proven in \cite{Bamberger1977} via comparison principles and under a mild integrability condition for the time derivative $\dt \beta(u)$; also see \cite{DiazThelin1994,Lindgren22,Otto1996} for related results.
We will call such solutions \emph{regular weak solutions} in the following.
The results mentioned above generalize quite naturally to networks, multiple dimensions, and more general boundary conditions; see the above references as well as \cite{AkagiStefanelli2010, AkagiStefanelli2014,AltLuckhaus1983,DiBenedettoShowalter1981,GrangeMignot1972,Roubicek} for some results in this direction.

In this paper, we study a fully discrete approximation scheme for \eqref{eq:1} which, up to the particular boundary conditions, is equivalent to the methods considered in \cite{Bamberger1977,Raviart1970}; also see 
\cite{SchoebelKroehn2020} for a related method. 
From the proofs in these references, we deduce the convergence of discrete approximations $u_{h,\tau}$ to weak solutions $u$ of \eqref{eq:1}
with mesh size $h$ and time step $\tau$ going to zero.
As a consequence, uniform a-priori estimates obtained for the discrete approximations $u_{h,\tau}$ immediately translate to corresponding bounds for weak solutions $u$ obtained by this limiting process.
In the current manuscript, we focus on the analysis of the underlying discretization scheme. 
Under suitable assumptions on initial and boundary data, which are motivated from the application in gas transport, we  establish
\begin{itemize}\itemsep1ex
\item uniform bounds $0 < \underline u \le u_{h,\tau} \le \overline u$ for the discrete solutions; 
\item  bounds on discrete time derivatives $\dtau u_{h,\tau}$ in $L^2$ independent of $h$ and $\tau$.
\end{itemize}
As an immediate consequence of these results, we obtain the \emph{existence of a regular weak solution $u$} on the continuous level; the arguments of \cite{Bamberger1977,DiazThelin1994}
then imply its \emph{uniqueness}.
The uniform bounds on $u$ further imply that $\beta$ is non-degenerate on the relevant range of values attained by the solution. Similar regularity and uniqueness results under non-degeneracy assumptions on $\beta$ were already established in \cite{DiazThelin1994} using different arguments.
The uniform bounds on $u$ also provide uniform bounds for the density $\rho=\beta^{-1}(u)$, which are of practical relevance in the gas transport context.

We continue our analysis of the discrete system by studying its \emph{long-term stability} using \emph{relative entropy estimates}. Our main results in this direction are
\begin{itemize}\itemsep1ex
\item exponential convergence of discrete solutions $u_{h,\tau}$ to the corresponding steady states, if the boundary data are kept constant;
\item uniform closeness of solutions $u_{h,\tau}$ to time-varying quasi-steady states corresponding to slowly varying boundary data.
\end{itemize}
These results are again uniform in the discretization parameters $h$ and $\tau$, and therefore carry over verbatim to solutions of the continuous problem.
Related stability results for the continuous problem were obtained in \cite{DiazThelin1994}, who however do not make use of uniform lower bounds for the densities stated above. Let us further refer to \cite{AndriyanovaMukminov2013,BhattacharyaMarazzi2016,BhattacharyaMarazzi2022} for some more recent results in this direction.

In contrast to previous work, we here investigate in detail the discretization scheme for \eqref{eq:1}, derive our main results for the fully discrete approximations, and show their uniformity with respect to the discretization parameters $h$ and $\tau$. Corresponding results on the continuous level are obtained as immediate consequences. 
Our analysis is based on a variational characterization of discrete solutions, monotonicity arguments, and relative energy estimates. 
Moreover, we make frequent use of the nonlinear transformation $\rho=\beta(u)$, which carries over to the discrete setting by appropriate design of the numerical approximation. 
Results of a similar flavor can be found in the context of \emph{entropy methods for diffusive partial differential equations}; see \cite{Juengel} for an introduction and further references.
Our main theorems are derived in detail for a single interval, but they generalize automatically to networks, if appropriate coupling conditions are prescribed at pipe junctions, as well as to multiple dimensions; see \cite{EggerGiesselmann2020,SchoebelKroehn2020} and \cite{AltLuckhaus1983,DiazThelin1994,Raviart1970} for results in these directions. Further possibilities for extension will be discussed at the end of the presentation.

\smallskip 

The remainder of the manuscript is organized as follows:
In Section~\ref{sec:prelim}, we present our notation and our main assumptions, 
%collect some preliminary results for the problem under consideration,
and introduce our discretization scheme.
The basic properties and a-priori estimates for the discrete solutions are summarized in Section~\ref{sec:properties}. 
Sections~\ref{sec:stability} and \ref{sec:stability2} are then devoted to the long-term stability of the discrete problem. 
In Section~\ref{sec:cont}, we summarize the implications for the continuous problem,
and we conclude with a short discussion.

\section{Preliminaries} \label{sec:prelim}

Let us start with a complete definition of our model problem. 
We consider the numerical approximation of the nonlinear parabolic equation
\begin{alignat}{2} 
    \dt \beta(u) - \dx ( \mu(\dx u)) &= 0, \qquad && 0 < x < \ell, \ t>0, \label{4a} \\
  n  \mu(\partial_x u ) + \alpha u &= \alpha u_\partial,  \qquad &&  x \in \{0,\ell\}, \, t>0. \label{4b}
\end{alignat}
The parameter $n$ in the boundary condition plays the role of a normal vector and takes the value $-1$ at $x=0$ and $1$ at $x=\ell$.
The system is complemented by the initial conditions 
\begin{alignat}{2}
u(0)=u_0, \qquad && 0 < x < \ell. \label{4c}
\end{alignat}
For our analysis, we impose the following conditions on the model parameters and problem data, which can be motivated by the physical background of the problem. 
\begin{assumption} \label{ass:1} 
$ $
\begin{itemize}\itemindent-1em
    \item $\alpha>0$ is a positive constant and $0< T< \infty$ denotes a time horizon;
    \item $\beta \in C^0(\RR)$ with $B(u) := \int_0^u \beta(s) \, ds$ coercive; furthermore $\beta|_{\RR\setminus 0} \in C^1$ with $\beta'>0$; 
    \item $\mu \in C^0(\RR)$, $\mu(0) = 0$,  $\mu|_{\RR\setminus 0} \in C^1$ with $\mu'>0$, and $M(s):= \int_0^s \mu(\tilde s) d\tilde s$; 
    \item 
    $u_0 \in L^\infty(0,\ell)$, such that $M(\dx u_0)\in L^1(0,\ell)$, $u_\partial \in W^{1,\infty}((0,\infty) \times \{0,\ell\})$;
    \item $0 < \underline u \le u_0,u_\partial \le \overline u$ uniformly in $x$ and $t$ with some constants $\overline u,\underline u>0$.
\end{itemize}
\end{assumption}
By our assumptions, $M$ is strictly convex and $\beta$ is strictly monotone with $\beta'(u) \ge \underline \beta'>0$ uniformly bounded from below on compact subsets $C \subset \RR\setminus 0$. 
Furthermore, the function $\eta(\rho) := \int_0^\rho \beta^{-1}(s) ds$, which will be of importance for our analysis later on, is well defined on all of $\RR$,  strictly convex and coercive. 
Also note that $\eta'(\rho)=u$ and $\rho=\beta(u)$. 

\subsection{Weak formulation and entropy dissipation}

By elementary computations, one can verify that any smooth solution of \eqref{4a}--\eqref{4b} 
satisfies the variational identity
\begin{align} \label{var4}
    \la \partial_t \beta(u), v \ra + \la \mu(\dx u), \dx v\ra + \alpha \la u,v\ra_\partial &= \alpha \la u_\partial(t), v\ra_\partial 
\end{align}
for all $ \ t> 0 \text{ and } v \in H^1(0,\ell) $.
For ease of notation, we use the symbols
\begin{align} \label{eq:sp}
\la f,g\ra = \int_0^\ell f g dx 
\qquad \text{and} \qquad 
\la f,g\ra_\partial = f(0) g(0) + f(\ell) g(\ell)
\end{align} 
to denote the scalar products over the domain and the boundary, respectively. 
For the following considerations, we further introduce the \emph{entropy} and \emph{dissipation}  functionals 
\begin{align} \label{eq:5}
\E(u) := \H(\beta(u)) := \int_0^\ell \eta(\beta(u)) dx,
\quad 
\D(u) := \int_0^\ell \mu(\dx u)  \dx u \, dx + \alpha \la u, u \ra_\partial,
\end{align}
which are of relevance for our stability analysis. 
By formally differentiating the entropy along a sufficiently smooth solution of \eqref{4a}--\eqref{4b}, we obtain 
\begin{align*}
\ddt \E(u) 
&= \la \dt \beta(u), \eta'(\beta(u)) \ra
 = -\la \mu(\dx u), \dx u\ra - \alpha \la u - u_\partial, u\ra_\partial \\
&= -\D(u)+ \alpha \la u_\partial,u \ra_\partial.
\end{align*}
This \emph{entropy identity} allows to derive suitable a-priori bounds for $u$ and $\dx u$ in terms of the problem data, which is a key ingredient for establishing the existence of weak solutions; see \cite{Raviart1970,SchoebelKroehn2020}.
A corresponding result will be stated in Section~\ref{sec:cont}.

\subsection{Discretization scheme} \label{sec:3}

For the numerical approximation of \eqref{4a}--\eqref{4c}, we use a suitable discretization of the variational principle \eqref{var4}.
Let $\T_h=\{0 = x_0 < x_1 < \ldots < x_N=\ell\}$ be a uniform mesh of the interval $(0,\ell)$ with grid points $x_i = i h$ and $h=\ell/N$, and let 
\begin{align*}
V_h = P_1(\T_h) \cap C[0,\ell],    
\end{align*}
denote the space of continuous piecewise linear finite element functions defined on this mesh. 
We further write $I_h : H^1(0,\ell) \to V_h$ for the piecewise linear nodal interpolation operator and introduce an approximation
\begin{align}
    \la a,b\ra_h = \int_0^\ell I_h(a b) \, dx
\end{align}
for the scalar product $\la a,b \ra=\int_0^\ell a b \, dx$ obtained through numerical integration by the trapezoidal rule.
For later reference, let us mention that 
\begin{align} \label{eq:equiv}
\la v_h,v_h \ra \le \la v_h,v_h\ra_h \le 3 \la v_h,v_h\ra \qquad \forall v_h \in V_h,
\end{align}
i.e., the discrete scalar product is equivalent to the exact one on the discrete space $V_h$.
By $t^k = k \tau$, $k \ge 0$, we denote the sequence of discrete time points with uniform time step $\tau>0$, and we further write $\dtau u^k = \frac{1}{\tau} (u^k - u^{k-1})$ for the backward difference quotient. 
We then consider the following discretization scheme.
\begin{problem}\label{prob2}
    Given $u_h^0 = I_h u^0$ find $u_h^k \in V_h$ for $k \ge 1$ such that
\begin{align} \label{eq:scheme2}
    \la \dtau \beta(u_h^k), v_h \ra_h + \la \mu(\dx u_h^k), \dx v_h\ra + \alpha \la u_h^k,v_h\ra_\partial &= \alpha \la u_\partial(t^k), v_h\ra_\partial  \quad
    \forall v_h \in V_h. 
\end{align}
\end{problem}
\noindent
This scheme combines an implicit Euler method in time and a finite-element approximation in space with numerical integration applied for the first term.
\begin{remark} \label{rem:equivalence}
Due to the use of the inexact scalar product $\la \cdot,\cdot\ra_h$, the nonlinear transformation $\beta(\cdot)$ only acts at the single values of $u_h^k$ at mesh vertices. 
Moreover, $\dx u_h^k$ is piecewise constant and, therefore $\mu(\cdot)$ again only acts on constant values.
The scheme \eqref{eq:scheme2} was thoroughly analyzed in \cite{SchoebelKroehn2020}. Up to boundary conditions, it is equivalent to the finite difference approximation used in \cite[Sec.~2]{Raviart1970} for the discretization of \eqref{4a} with Dirichlet boundary conditions. 
From the theoretical results of these papers, we may thus deduce convergence of the discrete approximations to weak solutions of the continuous problem. 
\end{remark}

\section{Well-posedness and uniform bounds} \label{sec:properties}

As a first part of our analysis, we establish well-posedness of our discretization scheme  and some elementary properties of its solutions. 

\begin{lemma}\label{lem:wp}
Let Assumption~\ref{ass:1} hold. Then Problem~\ref{prob2} has a unique solution $(u_h^k)_{k \ge 0}$. Moreover, the discrete entropy
$ 
\E_h(u_h):= \H_h(\beta(u_h)) := \la \eta(\beta(u_h)),1\ra_h
$
is dissipated, i.e.
  \[   \dtau \E_h(u_h^k) \leq - \D(u_h^k) + \alpha \la u_h^k, u_\partial(t^k)\ra.\]
\end{lemma}
\begin{proof}
Note that $u_h^k$ satisfies \eqref{eq:scheme2} if and only if it minimizes 
\begin{equation*}
  L(w_h):= \frac1{\Delta t}   \langle   B(w_h) -\beta(u_h^{k-1}) w_h ,1 \rangle_h + \langle M(\partial_x w_h), 1 \rangle + 
 \alpha  \langle \tfrac{1}{2} w_h^2 - u_\partial(t^k)w_h,1 \rangle_\partial\,. 
\end{equation*}
By Assumption 1, the functional $L(\cdot)$ is strictly convex and coercive on $V_h$, and therefore admits a unique minimizer. This ensures well-posedness of the discrete problem.
By convexity of the entropy density $\eta(\cdot)$ and using \eqref{eq:scheme2}, we further obtain that
\begin{align*}
    \dtau \E_h(u_h^k)
    &\le \la \dtau \beta(u_h^k), \eta'(\beta(u_h^k))\ra_h 
     = \la \dtau \beta(u_h^k), u_h^k\ra_h\\
    &= - \la \mu (\dx u_h^k), \dx u_h^k\ra 
    - \alpha \la u_h^k- u_\partial(t^k),u_h^k\ra_\partial .
\end{align*}
This yields the entropy inequality which was announced in the lemma. 
\end{proof}

\begin{remark}\label{re:aed}
    Note that the entropy dissipation can also be stated as
\begin{align*}
    \dtau \E(u_h^k) + \la \mu(\dx u_h^k), \dx u_h^k\ra + \alpha \la u_h^k- u_\partial(t^k), u_h^k- u_\partial(t^k)\ra_\partial
    \le \alpha  \la u_h^k- u_\partial(t^k),u_\partial(t^k) \ra_\partial.
\end{align*}
This will be used below to derive bounds for the discrete solutions $u_h^k$ in terms of the problem data $u_0$ and $u_\partial$ with uniform constants independent of $T$.
\end{remark}

We are now in the position to state and prove the first main result of our paper. 
\begin{theorem} \label{thm:main1}
Let Assumption~\ref{ass:1} hold and $(u_h^k)_{k \ge 0}$ denote the unique solution of Problem~\ref{prob2}. 
Then
$\underline u \le u_h^k \le \overline u$ for all $0 \le t^k \le T$ and
\begin{align*}
   \sum\nolimits_{t^k \le T} \tau \D(u_h^k) \le C
   \quad \text{and } \quad
      \sum\nolimits_{t^k \le T} \tau  \|\dtau u_h^k\|_{L^2(0,\ell)}^2  \le C
\end{align*}
with a constant $C=C(T)$ depending only on the problem data but independent of the discretization parameters $h$ and $\tau$. 
Moreover, $\int_0^\ell M(\dx u_h^k) dx \leq C$ for all $0 \leq t^k \le T$.
\end{theorem}

The remainder of this section is devoted to the proof of this result, which will be split into two steps. As mentioned in the introduction, the main points of concern here are the uniform bounds for the density $u_h^k$ and for the discrete time derivative $\dtau u_h^k$. The estimate for $\D(u_h^k)$ follows from entropy dissipation.

\subsection*{Uniform bounds} \label{sec:bounds}

As a first step, we establish the  uniform bounds for $u_h^k$; 
similar results have been announced in \cite{Bamberger1977,BambergerSorineYvon1979}.
By the mean value theorem, we may rewrite 
\begin{align*}
\mu(\dx u_h^k) 
&= \mu(\dx u_h^k) - \mu(\dx \underline u) 
= \mu'(\bar \omega_h^k) \dx (u_h^k - \underline u),
\end{align*}
where we used that $\dx \underline u=0$, $\mu(0)=0$, and the fact that $\dx u_h^k$ is piecewise constant. The mean values $\bar \omega_h^k$ are also piecewise constant and $\mu'(\bar \omega_h^k)$ is piecewise constant and positive, due to monotonicity and non-degeneracy of $\mu$. 
In a similar manner, we can expand 
\begin{align*}
I_h (\beta(u_h^k) - \beta(u_h^{k-1}))
&= (I_h \beta(u_h^k) - I_h \beta(\underline u)) - (I_h \beta(u_h^{k-1}) - I_h \beta(\underline u)) \\
&= I_h (\beta'(\xi_h^k) (u_h^k - \underline u)) - I_h (\beta'(\xi_h^{k-1}) (u_h^k - \underline u)   )
\end{align*}
with appropriate intermediate values $\xi_h^k$, $\xi_h^{k-1} \in V_h$.  
By using these formulas after subtracting the trivial identity
\begin{align*}
    \la \dtau \beta(\underline u), v_h \ra_h + \la \mu(\dx \underline u), \dx v_h\ra + \alpha \la \underline u,v_h\ra_\partial &= \alpha \la \underline u, v_h\ra_\partial  
\end{align*}
from \eqref{eq:scheme2}, we arrive at the linarized identity
\begin{align*}
\frac{1}{\tau} \la \beta'(\xi_h^k) (u_h^k &- \underline u), v_h \ra_h + \la \mu'(\bar \omega_h^k) \dx (u_h^k - \underline u), \dx v_h\ra + \alpha \la u_h^k - \underline u,v_h\ra_\partial\\
&= \frac{1}{\tau } \la \beta'(\xi_h^{k-1}) (u_h^{k-1} - \underline u), v_h \ra_h + \alpha \la u_\partial^k - \underline u, v_h\ra_\partial.  
\end{align*}
By choosing the standard nodal basis for $V_h$, we can transform this variational identity into an equivalent system of nonlinear equations 
\begin{align}\label{eq:mmatrix}
    [\tfrac{1}{\tau} D(\xi^k) + K(\bar \omega^k) + R] (u^k - \underline u) 
      &=  \tfrac{1}{\tau} D(\xi^{k-1}) (u^{k-1}-\underline u) + \alpha R (u_\partial^k - \underline u),
\end{align}
with $D(\xi)$ diagonal positive definite and $K(\bar \omega)$, $R$ weakly diagonally dominant $Z$-matrices ($A_{ii} \ge 0$, $A_{ij} \le 0$ for $i \ne j$, $\sum_j A_{ij} \ge 0$), and $u^k$ denoting the vector of nodal values of the function $u_h^k$, etc.  
From the properties stated above, we see that $[\tfrac{1}{\tau} D(\xi^k) + K(\bar \omega^k) + R]$ is a regular $M$-matrix and the vector on the right hand side of the above equation has non-negative entries. As a consequence the vector $u^k - \underline u$ has non-negative entries, which implies $u_h^k - \underline u \ge 0$. 
In a similar manner, one shows that $u_h^k - \overline u \le 0$.

\subsection*{Estimates for the time derivatives}

We now establish the bounds for 
$\dtau u_h^k$.  
%
% By basic calculations, we then get
% \begin{align*}
% \dtau (u^k -\bar u^k)^2
% &= \dtau (u^k)^2  - \dtau (\bar u^k)^2 - \dtau  (2\bar u^k) ( u^k - \bar u^k) \\
% & \qquad - 2\bar u^{k-1} (\dtau u^k - \dtau \bar u^k)\\
% &\le 2u^k \dtau u^k - 2\bar u^{k-1} \dtau \bar u^k - 2 \dtau \bar u^k ( u^k - \bar u^k) \\
% &\qquad - 2\bar u^{k-1}(\dtau u^k - \dtau \bar u^k)\\
% &= 2(u^k - \bar u^k) \dtau u^k + 2(\bar u^k - \bar u^{k-1}) \dtau u^k  %\\
% %&\qquad  
% -2\dtau \bar u^k (u^k - \bar u^k) \\
% &= 2(u^k - \bar u^k) \dtau u^k 
% - 2\dtau \bar u^k [(u^{k-1} - \bar u^{k-1}) - (\bar u^k - \bar u^{k-1})],
% \end{align*}
% which can be recast equivalently as 
% \begin{align*}
% (u^k - \bar u^k) \, \dtau u^k
% &\ge \frac12 \dtau (u^k-\bar u^k)^2 + \dtau \bar u^k (u^{k-1} - \bar u^{k-1}) - \dtau  \bar u^k \tau \dtau \bar u^k.  
% \end{align*}
%Then by 
%
By  testing the identity \eqref{eq:scheme2} with $v_h = \dtau u_h^k$, we obtain
\begin{align}
0 
&= \la \dtau \beta(u_h^k), \dtau u_h^k\ra_h + \la\mu(\dx u_h^k), \dx \dtau u_h^k\ra + \alpha \la u_h^k - u_\partial^k,\dtau u_h^k\ra_\partial \notag \\
&\ge \la\beta'(\xi_h^k) \dtau u_h^k, \dtau u_h^k\ra_h + \dtau \la M(\dx u_h^k), 1\ra +
\dtau \tfrac{\alpha}{2} | u_h^k - u_\partial^k|_\partial^2
- 
\alpha \la u_h^k - u_\partial^k,\dtau u_\partial^k\ra_\partial.  \label{eq:estdtu}
\end{align}
Here we used that $M(w)$ is a convex primitive of $\mu(w)$, i.e., $M'(w)=\mu(w)$. 
The absolute value of the last term can be bounded by our assumptions on $u_\partial$ and the uniform bounds for $u_h^k$.
Furthermore, we know that $\beta'(\xi_h^k) \ge c_3 > 0$ by strict monotonicity of $\beta$ and boundedness of $\xi_h^k$. 
Applying Young's inequality to the last term in \eqref{eq:estdtu}, using the equivalence \eqref{eq:equiv} of scalar products and Gronwall's lemma, we thus obtain
\begin{align*}
\la M(\dx u_h^k),1\ra &+ \frac{\alpha}{2} \la (u_h^k- u_\partial^k)^2,1\ra_\partial + c_3 \sum\nolimits_{j \le k} \tau \|\dtau u_h^j\|_{L^2(0,\ell)}^2 \\
&\le  
 C \Big( \la M(\dx u_h^0),1\ra + \frac{\alpha}{2} \la (u_h^0-u_\partial(t^0))^2,1\ra_\partial + \sum\nolimits_{j \le k} \tau |\dtau u_\partial(t^j)|_\partial \Big)
\end{align*}
with a constant $C$ that is independent of $h$ and $\tau$.
Using the conditions on the initial and boundary values in Assumption~\ref{ass:1}, we thus deduce the desired bounds for the discrete time derivatives $\dtau u_h^k$. 
%, uniform in the discretization parameters.
%
This completes the proof of Theorem~\ref{thm:main1}. \qed 

\begin{remark}\label{rem:cuniform}
  If $u_\partial \in L^2(0,\infty;\RR^2)$ then we obtain the statement of Theorem \ref{thm:main1} with $C$ independent of $T$. This can be seen as follows: 
 The entropy dissipation inequality as stated in Remark \ref{re:aed} shows that $u_\partial \in L^2(0,\infty;\RR^2)$ implies that 
  $\sum_{t^k \leq T} \tau \| u_h^k - u_\partial^k\|_\partial^2$ is  bounded independently of $T$.
  This allows to bound the last term in \eqref{eq:estdtu} without relying on Gronwall's lemma.
\end{remark}

\section{Exponential convergence to steady state} \label{sec:stability} 

In the second part of our analysis, we consider the stability of discrete solutions for time invariant or slowly varying boundary data. 
We begin by showing that  solutions of Problem~\ref{prob2} converge exponentially fast to steady states defined by 
\begin{align} \label{eq:5}
 \la \mu(\dx \hat u_h), \dx v_h\ra + \alpha \la \hat u_h,v_h\ra_\partial &= \alpha \la \hat u_\partial, v_h\ra_\partial  \quad
    \forall v_h \in V_h,
\end{align}
if the boundary data $u_\partial(t^k) = \hat u_\partial$ are kept constant for $t \ge t_0>0$. 
As a preparatory result, let us first state the well-posedness of the stationary problem.
\begin{lemma} \label{lem:stat}
Let Assumption~\ref{ass:1} hold.
Then for any boundary value $\hat u_\partial$ with $\underline u \le \hat u_\partial \le \overline u$, there exists a unique solution $\hat u_h \in V_h$ of \eqref{eq:5}. 
Moreover, the discrete solution $\dx \hat u_h$ is constant on $[0,\ell]$, $\underline u \le \hat u_h \le \overline u$, and $|\hat u_h|_{W^{1,\infty}(0,\ell)}  \le \frac{\overline u - \underline u}{\ell}$ with $C$ independent of $h$. 
\end{lemma}
\noindent
The results follow from the same arguments as used in the proofs of the previous section.
Let us note that the solution $\hat u_h$ of the discrete problem \eqref{eq:5} is also the unique exact solution of the stationary problem
\begin{align}
- \dx \mu(\dx \hat u) &=0  \label{eq:stat1}  \\
n \mu(\dx \hat u) + \alpha \hat u &= \alpha \hat u_\partial, \label{eq:stat2}   
\end{align}
which can be verified by elementary arguments.
We are now in the position to state the main results of this section about the exponential stability of nonlinear problem. 
\begin{theorem} \label{thm:main2}
Let Assumption~\ref{ass:1} hold,
% with $\alpha>0$, 
$(u_h^k)_{k \ge 0}$ denote the solution of Problem~\ref{prob2} with $u_\partial(t^k)=\hat u_\partial$ for all $t^k \ge t_0$ and $\hat u_h$ be the solution of \eqref{eq:5}. Then there exist positive constants $c_1^*,c_2^*$, independent of $h$ and $\tau$, such that 
\begin{align*}
    \|u_h^k - \hat u_h^{\vphantom{k}}\|_{L^2(0,\ell)} \le c_1^* e^{-c_2^* (t^k - t^j)} \|u_h^j - \hat u_h^{\vphantom{k}}\|_{L^2(0,\ell)} \qquad \forall t^k \ge t^j \ge 0.
\end{align*}
Moreover, the discrete time differences can be bounded uniformly by
$\sum_{t^k} \tau \| \dtau u_h^k\| ^2 \leq C$.
\end{theorem}
\begin{proof}
Since $u_h^k$ and $\hat u_h$ are uniformly bounded, we may assume $t_0 = 0$ without loss of generality. 
Let us define $\rho_h^k = I_h(\beta(u_h^k))$ and $\hat \rho_h = I_h(\beta(\hat u_h))$, and note that 
\begin{align*}
0 < \underline \rho \le \rho_h^k , \hat \rho_h \le \overline \rho, 
\end{align*}
for $\underline \rho = \beta(\underline u)$ and $\overline \rho = \beta(\overline u)$, which follows immediately from the previous results.
As a next step, we introduce the discrete relative entropy 
\begin{align} \label{eq:rel}
    \H_h(\rho|\hat \rho) = \la \eta(\rho),1\ra_h - \la \eta(\hat \rho),1\ra_h - \la \eta'(\hat \rho), \rho - \hat \rho\ra_h.
\end{align}
From the strict convexity of $\eta(\cdot)$, the uniform bounds for $\rho_h^k$ and $\hat \rho_h^{\vphantom{k}}$, and the equivalence of the scalar products $\la \cdot,\cdot\ra_h$ with $\la \cdot,\cdot\ra$ for discrete functions, we deduce that 
\begin{align} \label{eq:equivalence}
    c_1 \|\rho_h^k - \hat \rho_h^{\vphantom{k}}\|_{L^2(0,\ell)}^2 \le \H_h(\rho_h^k | \hat \rho_h^{\vphantom{k}}) \le c_2 \|\rho_h^k - \hat \rho_h^{\vphantom{k}}\|_{L^2(0,\ell)}^2
\end{align}
with uniform constants $c_1,c_2$ only depending on the bounds in Assumption~\ref{ass:1}. 
The first assertion of Theorem~\ref{thm:main2} now follows immediately from inequality
\begin{align} \label{eq:inequality}
\dtau \H_h(\rho_h^k|\hat \rho_h^{\vphantom{k}}) \le -c \H_h(\rho_h^k|\hat \rho_h^{\vphantom{k}}),
\end{align}
summation over the time-steps, and integration in time, and the use of \eqref{eq:equivalence}. 
The bound for the time differences follows with similar arguments as in Remark~\ref{rem:cuniform}.
\end{proof}

The remainder of this section is devoted to the proof of \eqref{eq:inequality} which, for the convenience of the reader, is again split into several steps. 

\subsection*{Relative entropy inequality}

From the definition of the discrete relative entropy, the convexity of $\eta(\cdot)$, and the variational problems \eqref{eq:scheme2} and \eqref{eq:5}, characterizing the discrete solutions $u_h^k$ and $\hat u_h$, we immediately deduce that 
\begin{align} \label{eq:relenin}
\dtau \H_h(\rho_h^k|\hat \rho_h) 
&\le \la \eta'(\rho_h^k) - \eta'(\hat \rho_h), \dtau \rho_h^k\ra_h
= \la u_h^k - \hat u_h, \dtau  \beta (u_h^k)\ra_h
\\ 
&= -\la  \mu (\dx u_h^k)- \mu(\dx \hat u_h), \dx(u_h^k - \hat u_h ) \ra  - \alpha \la u_h^k - \hat u_h, u_h^k - \hat u_h\ra_\partial. \notag 
\end{align}
It remains to show that the two dissipation terms on the right hand side can be bounded appropriately in terms of the discrete relative entropy. 
\subsection*{Bounds for the dissipation terms}
Let us recall from Assumption~\ref{ass:1} that 
$\mu$ is locally Lipschitz continuous and strictly monotonically increasing.
The following auxiliary result then allows us to bound the first dissipation term.
\begin{lemma}
For any $y,z \in \mathbb{R}$  with $|z| \le C $, there exists $c>0$ such that
\begin{align*} 
( \mu(y) - \mu(z)) (y-z)  \geq  c F(y,z)
\end{align*}
with dissipation function
\begin{align*}
F(y,z):= \left\{ \begin{array}{ccc}
| y-z |^2   & \text{ for }  |y - z | \leq 1,\\
|y-z|   & \text{ for }  |y - z | > 1.
\end{array}\right.
\end{align*}
\end{lemma}
\begin{proof}
The function  $\mu'(\cdot)$ is uniformly bounded from below on the interval $[-C-1,C+1]$ by a constant $\underline \mu'>0$. 
If $|y-z|\leq 1$, we thus have
$$
|\mu(y)-\mu(z)| \ge \underline \mu' |y-z|,
$$
which yields the first part of the assertion with $c=\underline \mu'$. 
For the second case, we choose $\tilde y$ between $y$ and $z$ with $|\tilde y-z| = 1$. By monotonicity and the previous bound, we see that
\begin{align*}
|\mu(y) - \mu(z)| = |\mu(y) - \tilde \mu(y)| + |\mu(\tilde y) - \mu(z)| \ge 0 + \underline \mu' |\tilde y - z| = \underline \mu'.
\end{align*}
This yields the estimate for the second case.
\end{proof}

\subsection*{An intermediate step}
By the auxiliary results derived so far, we arrive at 
\begin{align} \label{eq:intermediate}
\dtau \H_h(\rho_h^k|\hat \rho_h) 
\leq 
%- c  \int_0^\ell F(\dx u^k_h,\dx \hat u_h ) dx  -    |\eta'(\rho_h^k) -\eta'(\hat \rho_h)|_\partial^{2} .
- c \left( \int_0^\ell F(\dx u^k_h,\dx \hat u_h ) dx  +    |u_h^k - \hat u_h|_\partial^{2} \right).
\end{align}
Since $F \geq 0$, this already shows that
\[
c \sum_{t^k} \tau |u_h^k - \hat u|^2_\partial
\le \H(\rho^0_h, \hat \rho_h).
\]
which, by the arguments of Remark  \ref{rem:cuniform},
implies that we obtain the desired bound for $\dtau u_h^k$ uniformly in time.
\smallskip \noindent
By further bounding the right hand side in \eqref{eq:intermediate}, we arrive at the following estimate.
\begin{lemma}
Let $N(x)=x+x^{2}$ and $\nu$ be the inverse function of $N$. Then 
\begin{align*}
\dtau \H_h(\rho_h^k|\hat \rho_h) \le - c_1  \nu( c_2 \, \|u_h^k - \hat u_h\|_{L^2(0,\ell)}^2)
\end{align*}
with constants $c_1,c_2$ independent of $h$ and $\tau$.
\end{lemma}
\begin{proof}
We define $g(x):= u_h^k(x) - \hat u_h(x)$
and  decompose the domain $[0,\ell]$ into
\begin{align*} 
A := 
 \{ x \in [0,\ell] \, : \, |\dx u_h^k(x) - \dx \hat u_h(x)| \leq 1\} 
 \qquad \text{and} \qquad 
 A^c=[0,\ell]\setminus A.
\end{align*}
This allows us to express the integrand in the second term of \eqref{eq:intermediate} by 
\begin{equation}
   F(\dx u_h^k, \dx \hat u_h) = \left\{ \begin{array}{ccc}
  |\partial_x g(x) |^2 & \text{ on } & A,\\
  |\partial_x g(x) | & \text{ on } & A^c.
   \end{array}\right.
\end{equation}
For any $x \in [0,\ell]$, we can then further estimate
\begin{align*}
    g(x)^2 &\leq 2 g(0)^2 + 2\left( \int_0^x \partial_y g(y) dy \right)^2\\
    & \leq 2 g(0)^2  + 2 \left( \int_A   | \partial_y g(y)| dy   \right)^2  + 2 \left( \int_{A^c}   | \partial_y g(y)| dy   \right)^2\\
    & \leq 2 g(0)^2  + 2 \ell   \int_A | \partial_y g(y)|^2 dy    + 2   \left( \int_{A^c}  | \partial_y g(y)| dy    \right)^2\\
    & \leq c_\ell \left( g(0)^2  + \int_0^\ell F(\dx u^k_h,\dx  \hat u_h ) dx  + \left(\int_0^\ell  F(\dx u^k_h,\dx \hat u_h ) dx
    \right)^2 \right)
\end{align*} 
From the definition of $N(x)$, one can see that $a+N(b) \le N(a+b)$ for any $a,b \ge 0$. 
By integration over $x$ and using the definition of $g(x)$ and $N(x)$, we thus arrive at
\begin{equation}\label{eq:gL2bound}
    \| u_h^k - \hat u_h \|_{L^2}^2 \leq c_\ell \ell N\Big( \| u_h^k - \hat u_h\|_\partial^2 +  \int_0^\ell F(\dx u_h^k,\dx \hat u_h) dx\Big).
\end{equation}
By application of the inverse function $\nu(\cdot)$ on both sides of this inequality, we  obtain the assertion of the lemma.
\end{proof}

\subsection*{Proof of inequality \eqref{eq:inequality}}
By the uniform bounds for $\rho_h^k$, the strict convexity of $\eta(\cdot)$,  elementary bounds for the interpolation $I_h$, and \eqref{eq:equivalence}, we further see that
$$
\| u_h^k - \hat u_h \|_{L^2}^2 
 \geq c_3 \H_h(\rho_h^k|\hat \rho_h).
$$
Together with the previous lemma, we thus obtain 
\begin{equation}
  \dtau H_h(\rho_h^k|\hat \rho_h)   \leq -  c_1 \nu (c_4 \H_h(\rho_h^k|\hat \rho_h) ) 
\end{equation}
with positive constant $c_4$ still independent of $h$ and $\tau$. 
Since $\underline \rho \le \rho^k_h, \hat \rho_h \le \overline \rho$, we know that 
$0 \le \H_h(\rho_h^k|\hat \rho_h) \le C_H$ and 
note that $\nu(y) \ge c_5 y$ for all $y \in [0,C_H]$. This finally leads to 
\begin{equation}
  \dtau \H_h(\rho_h^k|\hat \rho_h)   \leq -  c  \H_h(\rho_h^k|\hat \rho_h)
\end{equation}
with a uniform constant $c>0$ independent of $h$ and $\tau$.

\hfill \qed

\medskip 
\noindent 
In summary, we have thus completed the proof of our second main result.

\section{Stability around slowly varying steady states} \label{sec:stability2}

As a final step of our analysis on the discrete level, we will now demonstrate that the solutions $u_h^k$ of the instationary problem always stay close to the quasi-steady states $\hat u_h^k$, i.e. solutions to \eqref{eq:5}  with the corresponding boundary data $\hat u_\partial^k = u_\partial(t^k)$. 
\begin{theorem} \label{thm:main3}
Let Assumption~\ref{ass:1} hold and let $(u_h^k)_{k \ge 0}$ be a solution of Problem~\ref{prob2} and $\hat u_h^k$, $k \ge 0$, be the solutions of \eqref{eq:5} with corresponding boundary data $\hat u_\partial^k = u_\partial(t^k)$. Then there exist positive constants
$c_1^*,c_2^*, \gamma, \gamma'$ such that
\begin{align*}
 \|u_h^k - \hat u_h^k\|_{L^2(0,\ell)}^2 \le c_1^* e^{-\gamma (t^k - t^{m})} \|u_h^{m} - \hat u_h^m\|_{L^2(0,\ell)}^2 + c_2^* \int_{t^m}^{t^k} e^{-\gamma'(t^k - t)} \|\dt u_\partial\|_\partial^2 \, dt.
\end{align*}
\end{theorem}

This shows that the transient solution $u_h^k$ stays close to the corresponding quasi-steady states $\hat u_h^k$, if the boundary data $u_\partial(t^k)$ vary slowly.  
The remainder of this section is devoted to the proof of Theorem~\ref{thm:main3} which, for ease of presentation, is again split into several steps.

=== HIER ===

\subsection*{Stability of stationary states}

As a preparatory result, we show that the steady states defined by \eqref{eq:5} depend stably on the boundary data. 
\begin{lemma}
Let Assumption~\ref{ass:1} hold 
%, $\alpha>0$ 
and $\hat u_h^k$ be the solution of \eqref{eq:5} for $\hat u^j_\partial = u_\partial(t^k)$. Then
\begin{align*}
    \|\hat u_h^k - \hat u_h^{k-1}\|^2_{L^2(0,\ell)} \le \frac23 \ell \|u_\partial^k - u_\partial^{k-1}\|_\partial^2.
\end{align*}
\end{lemma}
\begin{proof}
By subtracting the corresponding equations \eqref{eq:5} with data $\hat u_\partial^k$ and $\hat u_\partial^{k-1}$, testing with $v_h=\hat u_h^k - \hat u_h^{k-1}$, using monotonicity of $\mu$, and applying Young's inequality, we obtain
\begin{align*}
    \|\hat u_h^k - \hat u_h^{k-1}\|^2_{\partial} \le  \|u_\partial^k - u_\partial^{k-1}\|_\partial^2.
\end{align*}
Since $\hat u^k, \hat u^{k-1}$ are affine linear, see Lemma \ref{lem:stat}, the corresponding bound for the $L^2$-norm can now be obtained by elementary calculations. 
\end{proof}

% \subsection{Proof of Theorem~\ref{thm:main3}}

\subsection*{Relative entropy estimates}

Let $\rho_h^k= I_h(\beta(u_h^k))$, $\hat \rho_h^k= I_h(\beta(\hat u_h^k))$, and
 $\H_h(\cdot|\cdot)$ be  defined as in Section \ref{sec:stability}.
Then by simple expansion, we see that 
\begin{align*}
\H_h(\rho_h^k|\hat \rho_h^k) &- \H_h(\rho_h^{k-1}|\hat \rho_h^{k-1}) \\
&= \big[ \H_h(\rho_h^k|\hat \rho_h^k) - \H_h(\rho_h^{k-1}|\hat \rho_h^{k}) \big]
+ \big[ \H_h(\rho_h^{k-1}|\hat \rho_h^k) - \H_h(\rho_h^{k-1}|\hat \rho_h^{k-1}) \big] = (i)+(ii).
\end{align*}
From the inequality \eqref{eq:inequality} derived in the previous section, we know that 
\begin{align*}
(i) =  \H_h(\rho_h^k|\hat \rho_h^k) - \H_h(\rho_h^{k-1}|\hat \rho_h^{k}) 
\le -c \tau \H_h(\rho_h^{k}|\hat \rho_h^{k}) 
\end{align*}
with a uniform constant $c>0$; this already yields the required bound for the first term. 
In order to further estimate the second term, we use that
\begin{align*}
(ii) 
&= \H_h(\rho_h^{k-1}|\hat \rho_h^{k}) - \H_h(\rho_h^{k-1}|\hat \rho_h^{k-1}) \\
&= \la \eta(\rho_h^{k-1}),1\ra_h - \la \eta(\hat \rho_h^k),1\ra_h - \la \eta'(\hat \rho_h^k), \rho_h^{k-1} - \hat \rho_h^k\ra_h \\
& \qquad \qquad 
- \la\eta(\rho_h^{k-1}),1\ra_h + \la \eta(\hat \rho_h^{k-1}),1\ra_h + \la \eta'(\hat \rho_h^{k-1}) , \rho_h^{k-1} - \hat \rho_h^{k-1} \ra_h \\
&=\la \eta(\hat \rho_h^{k-1}) - \eta(\hat \rho_h^k) - \eta'(\hat \rho_h^{k-1})(\hat \rho_h^{k-1} - \hat \rho_h^k),1\ra_h 
- \la \eta'(\hat \rho_h^{k-1}) - \eta'(\hat \rho_h^k), \hat \rho_h^{k} - \hat \rho_h^{k-1} \ra_h  \\
& \qquad \qquad + \la \eta'(\hat \rho_h^{k-1}) - \eta'(\hat \rho_h^k), \rho_h^{k-1} - \hat \rho_h^{k-1}\ra_h  \\
&\le c_1 \|\hat \rho_h^{k} - \hat \rho_h^{k-1}\|_{L^2(0,\ell)}^2 + c_2 \|\hat \rho_h^k - \hat \rho_h^{k-1}\|_{L^2(0,\ell)} \|\rho_h^{k-1} - \hat \rho_h^{k-1}\|_{L^2(0,\ell)} 
\end{align*}
From the definitions of $\rho_h^k$ and $\hat \rho_h^k$, the uniform Lipschitz continuity of $\beta$ on $[\underline u, \overline u]$, and the estimates of the previous Lemma, we deduce that
\begin{align*} \|\hat \rho_h^{k} - \hat \rho_h^{k-1}\|_{L^2(0,\ell)}^2
\leq C
\|u_h^{k} - \hat u_h^{k-1}\|_{L^2(0,\ell)}^2
\le C \tau
\int_{t^{k-1}}^{t^k} \| \dt u_\partial \|_{\partial}^2
\end{align*}
By application of Young's inequality, we thus arrive at
\begin{align*}
(ii) 
    \le c_3 (c_1 \tau + c_2^2) \int_{t^{k-1}}^{t^k} \|\dt u_\partial\|_\partial^2 \, dt  +  \tau \tfrac{c_5}{2} \H_h(\rho_h^{k-1}|\hat \rho_h^{k-1})
\end{align*}
where we can make $c_5$ small and ensure $c_5 < c$.
By combination of the previous estimates, we finally obtain 
\begin{align*}
\H_h(\rho_h^k|\hat \rho_h^k) &-\H_h(\rho_h^{k-1}|\hat \rho_h^{k-1}) \\
&\le -c \tau \H_h(\rho_h^k|\hat \rho_h^k)   
+ \tau \frac{c_5}{2} \H_h(\rho_h^{k-1}|\hat \rho_h^{k-1}) 
+ c_4 \int_{t^{k-1}}^{t^k} \|\dt u_\partial\|_\partial^2 \, dt 
\end{align*}
Rearranging the terms then further leads to
\begin{align*}
    \H_h(\rho_h^k|\hat \rho_h^{k}) \le e^{-\gamma \tau} \H_h(\rho_h^{k-1}|\hat \rho_h^{k-1}) + c_1'  \|\dt u_\partial\|^2_{L^2(t^{k-1},t^k;\RR^2)},
\end{align*}
with $\gamma>0$, depending only on $c$ but independent of $\tau$. 
By recursive application of this inequality, we finally obtain the stability estimate 
\begin{align*}
\H_h(\rho_h^k|\hat \rho_h^k) 
%&\le e^{-\gamma (t^k - t^{m})}     \H_h(\rho_h^{m}|\hat \rho_h^m) + C \sum_{j=m+1}^k  e^{-\gamma(t^k-t^j)} \|\dt u_\partial(s)\|^2_{L^2(t^{j-1},t^j;\RR^2)} \, ds\\
&\le e^{-\gamma (t^k - t^{m})}     \H_h(\rho_h^{m}|\hat \rho_h^m) + C' \int_{t^m}^{t^k}  e^{-\gamma(t^k-s)} \|\dt u_\partial(s)\|_\partial^2 \, ds
\end{align*}

\subsection*{Proof of Theorem~\ref{thm:main3}}

By the uniform bounds for $u_h^k$, $\hat u_h^k$ and $\rho_h^k=\beta(u_h^k)$, $\hat \rho_h^k=\beta(\hat u_h^k)$, the equivalence statement \eqref{eq:equivalence}, and the conditions on $\beta$, we see that 
\begin{align*}
c' \| u_h^m - \hat u_h^m\| \le \H_h(\rho_h^{m}|\hat \rho_h^m)  \leq C' \| u_h^m - \hat u_h^m\|.
\end{align*}
By the previous estimates, we now obtain the assertion of Theorem~\ref{thm:main3}.  \qed
%

%%%%%%%%%%%%%%%%%%%%%%%%%%%%%%%%%%%%%%%%%%%%%%%%%%%%%%%%%%%%%%%%%%%%%
\section{Consequences for the continuous problem}\label{sec:cont}

As a last part of our analysis, we now draw some conclusions of our results obtained so far. 
For ease of presentation, we we assume a special form 
\begin{equation}\label{eq:mu}
\mu(s)= |s|^{p-2} s , \quad 1 < p \leq 2
\end{equation}
for the second model parameter. 
From the uniformity of the previous estimates and the convergence of discrete approximations to weak solutions, we immediately deduce
\begin{theorem} \label{thm:cont}
Let Assumption~\ref{ass:1} and condition \eqref{eq:mu} hold. Then \eqref{4a}--\eqref{4c} has a unique regular weak solution 
\begin{align*}
u &\in L^\infty(0,T;W ^{1,p}(0,\ell)) \cap H^1(0,T;L^2(0,\ell)) .
\end{align*}
Moreover $\underline u \le u \le \overline u$ on $(0,\ell) \times (0,T)$.
\end{theorem}
\begin{proof}
Existence of weak solutions to \eqref{4a} as limits of discrete approximations obtained via the problem \eqref{eq:scheme2} is proven in \cite{SchoebelKroehn2020}; see \cite{Raviart1970} for 
corresponding results with Dirichlet boundary conditions.
Theorem \ref{thm:main1} then implies the additional regularity and uniform bounds for  the solution. Under the established integrability condition for the time derivative, uniqueness of the weak solution is known; see \cite{DiazThelin1994, SchoebelKroehn2020}.
\end{proof}

From the assertions of Theorem \ref{thm:main3} and the convergence of discrete solutions, we may further conclude the following result about long-term stability of the continuous problem.

\begin{theorem} \label{thm:cont2}
Let Assumption~\ref{ass:1} and condition \eqref{eq:mu} hold. Further let $u$ be the regular weak solution of \eqref{4a}--\eqref{4c} and let $\hat u(t)$ be solutions of the stationary problem \eqref{eq:stat1}--\eqref{eq:stat2} with the corresponding boundary data $\hat u_\partial= u_\partial(t)$. Then
\begin{align*}
 \|u(t) - \hat u(t)\|_{L^2(0,\ell)}^2 \le c_1^* e^{-\gamma (t - s)} \|u(s) - \hat u(s)\|_{L^2(0,\ell)}^2 + c_2^* \int_{s}^{t} e^{-\gamma'(t - r)} \|\dt u_\partial\|_\partial^2\, dr
\end{align*}
\end{theorem}

This shows that the solution $u$ of the transient problem remains close to the corresponding quasi-steady states, if the boundary data vary slowly. If they are kept constant, then $u(t)$ converges exponentially fast to the corresponding steady state $\hat u$ with $t \to \infty$.

\section{Discussion} \label{sec:summary}

Let us note that the results derived in this paper and the arguments used for their proofs carry over immediately to problems on networks \cite{ EggerGiesselmannKunkelPhilippi,SchoebelKroehn2020} and in higher dimensions \cite{AkagiStefanelli2010,AltLuckhaus1983}, including the parabolic $p$-Laplacian (for $ 1 < p \leq 2$).
We expect that many results also carry over to problems with suitable lower order terms.
A key assumption for our analysis, motivated by applications in gas transport,  was the availability of uniform positive bounds for the initial and boundary data. Problems with vanishing or sign changing solutions have been discussed in \cite{DiazThelin1994} using different approximation techniques.
Further note that the stability results obtained in our analysis may also be useful for investigating controlability and observability of the underlying nonlinear system.
In the spirit of \cite{EggerGiesselmannKunkelPhilippi}, the discrete relative entropy estimates that were derived in this paper may serve as a first step for a rigorous discretization error analysis and the proof of convergence rates. 
This is left as a topic for future research.

{\footnotesize 
\section*{Acknowledgment}
The authors are grateful for financial support by the German Science Foundation (DFG) via grant TRR~154 (\emph{Mathematical modelling, simulation and optimization using the example of gas networks}), project C05 and the \emph{Center for Computational Engineering} at Technische Universität Darmstadt.
}

\nocite{*}

 \bibliographystyle{abbrv}
 \bibliography{parabolic}

\begin{thebibliography}{10}

\bibitem{AiczicoviciHikkanen2004}
S.~Aizicovici and V.-M. Hokkanen.
\newblock Doubly nonlinear equations with unbounded operators.
\newblock {\em Nonlinear Anal.}, 58:591--607, 2004.

\bibitem{AkagiStefanelli2010}
G.~Akagi and U.~Stefanelli.
\newblock A variational principle for doubly nonlinear evolution.
\newblock {\em Appl. Math. Lett.}, 23:1120--1124, 2010.

\bibitem{AkagiStefanelli2014}
G.~Akagi and U.~Stefanelli.
\newblock Doubly nonlinear equations as convex minimization.
\newblock {\em SIAM J. Math. Anal.}, 46:1922--1945, 2014.

\bibitem{AltLuckhaus1983}
H.~W. Alt and S.~Luckhaus.
\newblock Quasilinear elliptic parabolic differential equations.
\newblock {\em Math. Z.}, 183:311--341, 1983.

\bibitem{AndriyanovaMukminov2013}
E.~R. Andriyanova and F.~K. Mukminov.
\newblock Stabilization of the solution of a parabolic equation with double
  nonlinearity.
\newblock {\em Mat. Sb.}, 204:3--28, 2013.

\bibitem{Bamberger1977}
A.~Bamberger.
\newblock \'{E}tude d'une \'{e}quation doublement non lin\'{e}aire.
\newblock {\em J. Functional Analysis}, 24:148--155, 1977.

\bibitem{BambergerSorineYvon1979}
A.~Bamberger, M.~Sorine, and J.~P. Yvon.
\newblock Analyse et contr\^{o}le d'un r\'{e}seau de transport de gaz.
\newblock In {\em Computing methods in applied sciences and engineering},
  volume~91 of {\em Lecture Notes in Phys.}, pages 347--359, 1979.

\bibitem{Benilan1972}
P.~B\'{e}nilan.
\newblock Solutions int\'{e}grales d'\'{e}quations d'\'{e}volution dans un
  espace de {B}anach.
\newblock {\em C. R. Acad. Sci. Paris S\'{e}r. A-B}, 274:A47--A50, 1972.

\bibitem{BhattacharyaMarazzi2016}
T.~Bhattacharya and L.~Marazzi.
\newblock Asymptotics of viscosity solutions to some doubly nonlinear parabolic
  equations.
\newblock {\em J. Evol. Equ.}, 16(4):759--788, 2016.

\bibitem{BhattacharyaMarazzi2022}
T.~Bhattacharya and L.~Marazzi.
\newblock A strong minimum principle and large time asymptotics for viscosity
  solutions to a class of doubly nonlinear possibly degenerate parabolic
  equations.
\newblock {\em Nonlinear Anal.}, 221:Paper No. 112875, 25, 2022.

\bibitem{BrouwerGasserHerty2011}
J.~Brouwer, I.~Gasser, and M.~Herty.
\newblock Gas pipeline models revisited: model hierarchies, nonisothermal
  models, and simulations of networks.
\newblock {\em Multiscale Model. Simul.}, 9:601--623, 2011.

\bibitem{BurlacuEtAl2019}
R.~Burlacu, H.~Egger, M.~Gro\ss, A.~Martin, M.~E. Pfetsch, L.~Schewe,
  M.~Sirvent, and M.~Skutella.
\newblock Maximizing the storage capacity of gas networks: a global {MINLP}
  approach.
\newblock {\em Optim. Eng.}, 20:543--573, 2019.

\bibitem{DiazThelin1994}
J.~I. {Diaz} and F.~{de Thelin}.
\newblock {On a nonlinear parabolic problem arising in some models related to
  turbulent flows}.
\newblock {\em {SIAM J. Math. Anal.}}, 25:1085--1111, 1994.

\bibitem{DiBenedettoShowalter1981}
E.~DiBenedetto and R.~E. Showalter.
\newblock Implicit degenerate evolution equations and applications.
\newblock {\em SIAM J. Math. Anal.}, 12:731--751, 1981.

\bibitem{DroniouEymard2016}
J.~Droniou and R.~Eymard.
\newblock Uniform-in-time convergence of numerical methods for non-linear
  degenerate parabolic equations.
\newblock {\em Numer. Math.}, 132:721--766, 2016.

\bibitem{DroniouEymardHerbin2016}
J.~Droniou, R.~Eymard, and R.~Herbin.
\newblock Gradient schemes: generic tools for the numerical analysis of
  diffusion equations.
\newblock {\em ESAIM Math. Model. Numer. Anal.}, 50:749--781, 2016.

\bibitem{EggerGiesselmann2020}
H.~Egger and J.~Giesselmann.
\newblock Stability and asymptotic analysis for instationary gas transport via
  relative energy estimates.
\newblock {\em Numer. Math.}, 153(4):701--728, 2023.

\bibitem{EggerGiesselmannKunkelPhilippi}
H.~Egger, J.~Giesselmann, T.~Kunkel, and N.~Philippi.
\newblock {An asymptotic-preserving discretization scheme for gas transport in
  pipe networks}.
\newblock {\em IMA Journal of Numerical Analysis}, 2022.
\newblock drac032.

\bibitem{GrangeMignot1972}
O.~Grange and F.~Mignot.
\newblock Sur la r\'{e}solution d'une \'{e}quation et d'une in\'{e}quation
  paraboliques non lin\'{e}aires.
\newblock {\em J. Functional Analysis}, 11:77--92, 1972.

\bibitem{HyndLindgren2021}
R.~Hynd and E.~Lindgren.
\newblock Large time behavior of solutions of {T}rudinger's equation.
\newblock {\em J. Differential Equations}, 274:188--230, 2021.

\bibitem{Ivanov1997}
A.~V. Ivanov.
\newblock Regularity for doubly nonlinear parabolic equations.
\newblock {\em Zap. Nauchn. Sem. S.-Peterburg. Otdel. Mat. Inst. Steklov.
  (POMI)}, 209:37--59, 261, 1994.

\bibitem{Juengel}
A.~J\"ungel.
\newblock {\em Entropy methods for diffusive partial differential equations}.
\newblock Springer Briefs in Mathematics. Springer, 2016.

\bibitem{Korpusov2013}
M.~O. Korpusov.
\newblock On the blow-up of solutions of a class of parabolic equations with
  double nonlinearity.
\newblock {\em Mat. Sb.}, 204:19--42, 2013.

\bibitem{LanglaisPhillips1985}
M.~Langlais and D.~Phillips.
\newblock Stabilization of solutions of nonlinear and degenerate evolution
  equations.
\newblock {\em Nonlinear Anal.}, 9:321--333, 1985.

\bibitem{Lindgren22}
E.~Lindgren and P.~Lindqvist.
\newblock On a comparison principle for {T}rudinger's equation.
\newblock {\em Adv. Calc. Var.}, 15:401--415, 2022.

\bibitem{NovruzovHagverdiyev2016}
E.~Novruzov and A.~Hagverdiyev.
\newblock On long-time dynamics of the solution of doubly nonlinear equation.
\newblock {\em Qual. Theory Dyn. Syst.}, 15:127--155, 2016.

\bibitem{Otto1996}
F.~Otto.
\newblock {$L^1$}-contraction and uniqueness for quasilinear elliptic-parabolic
  equations.
\newblock {\em J. Differential Equations}, 131:20--38, 1996.

\bibitem{Plemmons77}
R.~J. Plemmons.
\newblock {$M$}-matrix characterizations. {I}. {N}onsingular {$M$}-matrices.
\newblock {\em Linear Algebra and Appl.}, 18:175--188, 1977.

\bibitem{Raviart1970}
P.~A. {Raviart}.
\newblock {Sur la r\'esolution de certaines \'equations paraboliques non
  lin\'eaires}.
\newblock {\em {J. Funct. Anal.}}, 5:299--328, 1970.

\bibitem{Roubicek}
T.~Roub{\'i}{\v c}ek.
\newblock {\em Nonlinear partial differential equations with applications},
  volume 153 of {\em International Series of Numerical Mathematics}.
\newblock Birkh\"auser/Springer Basel AG, Basel, second edition, 2013.

\bibitem{ScarpaStefanelli2020}
L.~Scarpa and U.~Stefanelli.
\newblock Doubly nonlinear stochastic evolution equations.
\newblock {\em Math. Models Methods Appl. Sci.}, 30:991--1031, 2020.

\bibitem{SchoebelKroehn2020}
L.~Sch\"obel-Kr\"ohn.
\newblock {\em Analysis and numerical approximation of nonlinear evolution
  equations on network structures}.
\newblock Dr. Hut Verlag, München, 2020.

\bibitem{Tsutsumi1988}
M.~Tsutsumi.
\newblock On solutions of some doubly nonlinear degenerate parabolic equations
  with absorption.
\newblock {\em J. Math. Anal. Appl.}, 132:187--212, 1988.

\end{thebibliography}
\end{document}